\documentclass[12pt,oneside,english]{amsart}
\usepackage[T1]{fontenc}
\usepackage[latin9]{inputenc}
\usepackage{fullpage}
\usepackage{color}
\usepackage{babel}
\usepackage{amstext}
\usepackage{amsthm}
\usepackage{amssymb}
\usepackage{verbatim, cancel}
\usepackage[unicode=true,pdfusetitle,
 bookmarks=true,bookmarksnumbered=false,bookmarksopen=false,
 breaklinks=false,pdfborder={0 0 1},backref=false,colorlinks=false]
 {hyperref}

\makeatletter
\theoremstyle{plain}
\newtheorem{thm}{\protect\theoremname}
\theoremstyle{definition}
\newtheorem{defn}[thm]{\protect\definitionname}
\theoremstyle{plain}
\newtheorem{prop}[thm]{\protect\propositionname}
\theoremstyle{remark}
\newtheorem{rem}[thm]{\protect\remarkname}
\theoremstyle{plain}
\newtheorem{cor}[thm]{\protect\corollaryname}
\theoremstyle{plain}
\newtheorem{lem}[thm]{\protect\lemmaname}
\theoremstyle{remark}
\newtheorem*{rem*}{\protect\remarkname}

\makeatother

\providecommand{\corollaryname}{Corollary}
\providecommand{\definitionname}{Definition}
\providecommand{\lemmaname}{Lemma}
\providecommand{\propositionname}{Proposition}
\providecommand{\remarkname}{Remark}
\providecommand{\theoremname}{Theorem}
\newcommand{\dx}{\mathrm{d}}
\newcommand{\E}{\mathfrak{E}}
\newcommand{\I}{\mathfrak{I}}
\newcommand{\M}{\mathfrak{M}}

\begin{document}
\pagestyle{plain}
\title{Laplace convolutions of weighted averages of arithmetical functions}
\author[M.~Cantarini, A.~Gambini, A.~Zaccagnini]
       {Marco Cantarini, Alessandro Gambini, Alessandro Zaccagnini}
\begin{abstract}
Let $G(g;x):=\sum_{n\leq x}g(n)$ be the summatory function of an arithmetical
function $g(n)$. In this paper, we prove that we can write weighted
averages of an arbitrary fixed number $N$ of arithmetical functions
$g_{j}(n),\,j\in\left\{ 1,\dots,N\right\} $ as an integral involving
the convolution (in the sense of Laplace) of $G_{j}(x),\,j\in\left\{ 1,\dots,N\right\} $.
Furthermore, we prove an identity that allows us to obtain known results
about averages of arithmetical functions in a very simple and natural
way, and overcome some technical limitations for some well-known problems.
\end{abstract}

\maketitle

\section{introduction}

The study of properties of arithmetical functions has a central role
in number theory. Many important problems can be reformulated in terms
of some kind of property that a particular arithmetical function must
verify. Maybe, one of the most famous example is that the Riemann
Hypothesis (RH) is equivalent to the formula
\[
\psi\left(x\right)=x+O\left(x^{1/2}\log^{2}\left(x\right)\right)
\]
as $x\rightarrow+\infty$, where $\psi\left(x\right):=\sum_{n\leq x}\Lambda\left(n\right)$
and $\Lambda\left(n\right)$ is the von Mangoldt function, defined
as
\[
\Lambda\left(n\right):=\begin{cases}
\log\left(p\right), & \text{if } n=p^{m}\,\text{for some prime }p\text{ and some positive integer }m\\
0, & \text{otherwise}
\end{cases}
\]
(see, e.g., \cite{Dav}, p. 113-114). As any number theorist knows,
the pointwise control of an irregular arithmetical function may
be very difficult, so it is natural to study averages of such functions,
since it could allow us to obtain some information about the behaviour
of such sequence. This idea, obviously, can be extended to general
additive problems. For example, if $g_{1}(n)$ and $g_{2}(n)$ are
two characteristic functions of some subsets of the natural numbers,
the sum
\begin{equation}
R_{g_{1},g_{2}}\left(N\right)
:=
\sum_{n=1}^{N}g_{1}\left(n\right)g_{2}\left(N-n\right)\label{eq:discreteconv}
\end{equation}
counts the number of representations of $N$ as a sum of elements that
belong to such subsets. It is not difficult to recognise, if $g_{1}(n)=g_{2}(n)=\Lambda(n)$,
that the sign of (\ref{eq:discreteconv}) is linked to the very famous
Goldbach binary problem. Clearly, if we are not able to control the
behaviour of $R_{g_{1},g_{2}}\left(N\right)$ it is quite natural to consider
the average $\sum_{n\leq x}R_{g_{1},g_{2}}\left(n\right)$ to try
to obtain some kind of information, or even to insert a smooth weight
in the sum, to expand the tools that can be used to study these problems.

There is an extremely wide literature about results in terms of
(possibly weighted) averages of arithmetical functions and it is quite
natural to continue to search new ways to attack these type of
problems. Motivated by these intentions, in this work we provide a
general approach for the evaluation of asymptotic formulas of averages
of arithmetical functions; in particular we focused on averages of a
general additive problem.
One of the main points of this paper is that, in some sense, the information
of the weighted averages of some arithmetical functions is controlled
by the convolution of their summatory functions. If we have enough
control of such a convolution, we can ``transport'' the information
in every weighted average, assuming that the weight is sufficiently
regular.

In this paper, we present a new general identity for the weighted
average of arithmetical functions in terms of integrals that resemble Laplace convolutions: see e.g. \eqref{def-conv} below. We show that this formula is indeed interesting due to its generality and versatility. Indeed, we are able to prove
that some known results can be achieved in a very natural, fast and
elegant way using our method, and some technical barriers can be overcome,
as we will show in the last part of the article for a very well-known
additive problem.

It is important to underline that the proposed examples are not the
main part of our work; we believe such examples just exemplify the
versatility of the main formula because our aim is to show that the
principal formula can be used in different problems.

A last observation is that, for technical reasons, we have to assume
some hypotheses on the weight $f$ but we are skeptical about the fact
that they are optimal. This idea comes from the fact that the
``integral transforms tool'' can be defined in many general settings
so it would be interesting to investigate possible generalisations of our formula.
Indeed, in the last chapter of the paper we show that our formula can
be interpreted in the language of the distributions and we demonstrate
how, with this language, it is easy to prove certain known theorems in
an extremely simpler and natural way. Again, also in this chapter, the
choice of what to prove is only exemplary, what we are really
interested in is showing how the main theorem of this article allows
different interpretations that will have to be studied in depth in the
future to obtain more profound results.

\subsection{Outline of the method}

Let $G(g;x):=\sum_{n\leq x}g(n)$ be the summatory function of an arithmetical
function $g(n).$ For the rest of the paper, we will simply write $G(x):=G(g;x)$ or $G_j(x):=G(g_j;x)$ to avoid excessive notation if this does not create ambiguity. In the first part we prove that if we take a sufficiently
regular weight $f$, a natural number $d\geq2$ and a positive real
number $\lambda,$ then the average 
\[
  \sum_{n_{1}}\cdots\sum_{n_{d}}g_{1}(n_{1})\cdots g_{d}(n_{d})f
  \left(\frac{n_{1}+\dots+n_{d}}{\lambda}\right)
\]
can be written in terms of an integral involving the convolution 
$\left(G_{1}*\left(G_{2}*\dots*\left(G_{d-1}*G_{d}\right)\right)\right)$
in the sense of Laplace. This is the main result of the paper.
In order to give an idea of our results, we give conditions on the
function $f$ so that we have the identity
\[
  \sum_{n + m \le N}
    g_1(n) g_2(m) f(n + m)
  =
  \int_{0}^{N} f''(w) (G_{1} * G_{2}) (w) \, \dx w,
\]
or more general versions of this relation; see, e.g.,
Proposition \ref{prop:abel sum}.

\subsection{Applications}
In the second part we give several applications of the general formulae
that we obtain in the first part.
In particular, we prove that an asymptotic formula for the convolution
$\psi*\psi$ leads to a general asymptotic formula for the weighted
averages of the Goldbach numbers.
We show that such a formula allows us to find very easily
some well-known results, like the behaviour of the Ces\`aro averages
of the Goldbach numbers with Ces\`aro weight of any order $k>0$, or
some properties of the Dirichlet series
$\sum_{n\geq1}\frac{R\left(n\right)}{n^{s}}$.
Here we rely on several results from the papers by Languasco \&
Zaccagnini \cite{LANZAC2} and by Br\"udern, Kaczorowski \& Perelli
\cite{BKP}.

We also address a problem studied by Languasco \& Zaccagnini in
\cite{LANZAC1}, namely an asymptotic formula with many terms for the
Ces\`aro average of the number of representations of an integer as a
sum of a prime and a perfect square.
In Corollary \ref{cor-HL-gen} we show that our main result, which is
Theorem \ref{prime_integ}, implies an asymptotic formula similar to
the one in \cite{LANZAC1} for the more general problem with a perfect
power in place of just a square.

\section{main identity}

In this section we will prove a general formula that resembles the
Abel summation formula in two dimensions (see, e.g., \cite{BAPO}
and \cite{KITA}, Lemma 3.5); then, we specialise this result using
the von Mangoldt function $\Lambda(n)$ and, obviously, linking the
problem to the well-known Goldbach's conjecture. Firstly, we introduce
the following definition.

\begin{defn}
Given a finite sequence of arithmetic functions $g_{j}:\mathbb{N}\rightarrow\mathbb{R}_{0}^{+}$,
$j=1,\dots,n,\,n\geq2$ we define the functions $G_{j},\,j=1,\dots,n$
as
\[
G_{j}\left(x\right):=\begin{cases}
\sum_{n\leq x}g_{j}\left(n\right), & \text{if } x>0\\
0, & \text{otherwise}.
\end{cases}
\]
\end{defn}

Now we are able to prove the following summation formula. 
\begin{prop}
\label{prop:abel sum}Let $f:\mathbb{R}\rightarrow\mathbb{C}$ and
assume that:

- $f$ has compact support on $\left[a,b\right]$, with $a<b$;

- $f\in C^{1}\left(a,b\right)$;

- $f^{\prime}$ is absolutely continuous on $\left(a,b\right)$;

- Both $f\left(a^{+}\right)$ and $f^{\prime}\left(a^{+}\right)$
exist and are finite and $f\left(b^{-}\right)=f^{\prime}\left(b^{-}\right)=0$.

Furthermore let $g_{1},g_{2}$ be two arithmetic functions and $\lambda>0.$
Then the following summation formula holds:
\begin{align}
  \sum_{\lambda a<n\leq\lambda b}g_{2}(n)
  \sum_{m\leq\lambda b-n}g_{1}
  \left(m\right)f\left(\frac{n+m}{\lambda}\right) 
  & =
  G_{2}\left(\lambda a\right)\left[\int_{a}^{b}G_{1}\left(\lambda v-\lambda a\right)f^{\prime}\left(v\right) \, \dx v\right]
  \label{eq:abel conseq}\\
 &\quad +
 \frac{1}{\lambda}\int_{a}^{b}f^{\prime\prime}\left(w\right)\left[\int_{\lambda a}^{\lambda w}G_{2}(s)G_{1}\left(\lambda w-s\right) \, \dx s\right] \, \dx w.\nonumber 
\end{align}
\end{prop}

\begin{proof}
Let $n\in\left(\lambda a,\lambda b\right]$. Since $f$ has compact
support on $[a,b]$ and since $f\left(b^{-}\right)=0$ we get, by the
classical Abel's summation formula (see, e.g., \cite{APO}, Theorem
4.2), that
\[
\sum_{m\leq\lambda b-n}g_{1}
\left(m\right)f\left(\frac{n+m}{\lambda}\right)
=
-\frac{1}{\lambda}\int_{0}^{\lambda b-n}G_{1}\left(t\right)f^{\prime}\left(\frac{n+t}{\lambda}\right) \, \dx t.
\]
Now, multiplying both sides by $g_{2}(n)$ and summing over
all $n\in\left(\lambda a,\lambda b\right]$, we obtain
\begin{equation}
\sum_{\lambda a<n\leq\lambda b}g_{2}(n)
\sum_{m\leq\lambda b-n}g_{1}
\left(m\right)f\left(\frac{n+m}{\lambda}\right)=-\frac{1}{\lambda}\sum_{\lambda a<n\leq\lambda b}g_{2}(n)\int_{0}^{\lambda b-n}G_{1}\left(t\right)f^{\prime}\left(\frac{n+t}{\lambda}\right) \, \dx t.\label{eq:nuoveformul1}
\end{equation}
Again by the Abel summation formula, we can rewrite the RHS of (\ref{eq:nuoveformul1})
as
\begin{align*}
  -\frac{1}{\lambda}\sum_{\lambda a<n\leq\lambda b}g_{2}(n)
  &\int_{0}^{\lambda b-n}G_{1}\left(t\right)f^{\prime}\left(\frac{n+t}{\lambda}\right) \dx t
  =
  \frac{1}{\lambda}G_{2}\left(\lambda a\right)\left[\int_{0}^{\lambda b-\lambda a}G_{1}\left(t\right)f^{\prime}\left(\frac{\lambda a+t}{\lambda}\right) \, \dx t\right]\\
 & \quad +
 \frac{1}{\lambda}\int_{\lambda a}^{\lambda b}G_{2}(s)\frac{\dx}{\dx s}\left[\int_{0}^{\lambda b-s}G_{1}\left(t\right)f^{\prime}\left(\frac{s+t}{\lambda}\right) \, \dx t \right] \, \dx s.
\end{align*}
Now we observe, from the Leibniz integral rule (see \cite{Fol2},
Theorem $2.27$) and from $f^{\prime}\left(b^{-}\right)=0$, that
\[
\frac{\dx}{\dx s}\left[\int_{0}^{\lambda b-s}G_{1}\left(t\right)f^{\prime}\left(\frac{s+t}{\lambda}\right) \, \dx t\right]
=
\frac{1}{\lambda}\int_{0}^{\lambda b-s}G_{1}\left(t\right)f^{\prime\prime}\left(\frac{s+t}{\lambda}\right) \, \dx t.
\]
Hence we have
\begin{align*}
  -\frac{1}{\lambda}\sum_{\lambda a<n\leq\lambda b}g_{2}(n)
  &
  \int_{0}^{\lambda b-n}G_{1}\left(t\right)f^{\prime}\left(\frac{n+t}{\lambda}\right) \, \dx t
  =
  \frac{1}{\lambda}G_{2}\left(\lambda a\right)\left[\int_{0}^{\lambda b-\lambda a}G_{1}\left(t\right)f^{\prime}\left(\frac{\lambda a+t}{\lambda}\right) \, \dx t\right]\\
 & \quad
 +\frac{1}{\lambda^{2}}\int_{\lambda a}^{\lambda b}G_{2}(s)\int_{0}^{\lambda b-s}G_{1}\left(t\right)f^{\prime\prime}\left(\frac{s+t}{\lambda}\right)\, \dx t \, \dx s.
\end{align*}
Putting $t=\lambda v-\lambda a$ in the first integral
and $t=\lambda w-s$ in the second, we obtain the identities
\[
\frac{1}{\lambda}G_{2}\left(\lambda a\right)\left[\int_{0}^{\lambda b-\lambda a}G_{1}\left(t\right)f^{\prime}\left(\frac{\lambda a+t}{\lambda}\right) \, \dx t\right]
=
G_{2}\left(\lambda a\right)\left[\int_{a}^{b}G_{1}\left(\lambda v-\lambda a\right)f^{\prime}\left(v\right) \, \dx v\right]
\]
and
\[
\frac{1}{\lambda^{2}}\int_{\lambda a}^{\lambda b}G_{2}(s)\int_{0}^{\lambda b-s}G_{1}\left(t\right)f^{\prime\prime}\left(\frac{s+t}{\lambda}\right) \, \dx t \, \dx s=\frac{1}{\lambda}\int_{\lambda a}^{\lambda b}G_{2}(s)\int_{\frac{s}{\lambda}}^{b}G_{1}\left(\lambda w-s\right)f^{\prime\prime}\left(w\right) \, \dx w \, \dx s
\]
\[
=\frac{1}{\lambda}\int_{\lambda a}^{\lambda b}G_{2}(s)\int_{a}^{b}G_{1}\left(w\lambda-s\right)f^{\prime\prime}\left(w\right) \, \dx w \, \dx s
=
\frac{1}{\lambda}\int_{a}^{b}f^{\prime\prime}\left(w\right)\int_{\lambda a}^{\lambda b}G_{2}(s)G_{1}\left(\lambda w-s\right) \, \dx s \, \dx w,
\]
where the second identity follows from the fact that $G_{1}\left(w\lambda-s\right)\equiv0$
if $w<\frac{s}{\lambda}$ and the last identity by Fubini's theorem.
So, finally, we can write that
\begin{align*}
\sum_{\lambda a<n\leq\lambda b}g_{2}(n)\sum_{m\leq\lambda b-n}g_{1}\left(m\right)f\left(\frac{n+m}{\lambda}\right) 
& =
\frac{1}{\lambda}G_{2}\left(\lambda a\right)\left[\int_{0}^{\lambda b-\lambda a}G_{1}\left(t\right)f^{\prime}\left(\frac{\lambda a+t}{\lambda}\right) \, \dx t\right]\\
 & \quad+
 \frac{1}{\lambda}\int_{a}^{b}f^{\prime\prime}\left(w\right)\int_{\lambda a}^{\lambda b}G_{2}(s)G_{1}\left(\lambda w-s\right) \, \dx s \, \dx w
\end{align*}
and now it remains to observe that 
\begin{align*}
\int_{\lambda a}^{\lambda b}G_{2}(s)G_{1}\left(w\lambda-s\right) \, \dx s 
& =
\int_{\lambda a}^{\lambda w}G_{2}(s)G_{1}\left(w\lambda-s\right) \, \dx s+\int_{\lambda w}^{\lambda b}G_{2}(s)G_{1}\left(\lambda w-s\right) \, \dx s\\
 & =
 \int_{\lambda a}^{\lambda w}G_{2}(s)G_{1}\left(\lambda w-s\right) \, \dx s
\end{align*}
since $G_{1}\left(w\lambda-s\right)\equiv0$ if $s\geq w\lambda$.
\end{proof}
\begin{rem}
\label{rem: variante}It is interesting to note that the previous
formula is quite versatile in the following sense: from the previous
proof we can obtain some new version of the main formula with small
modifications. For example, in the case $a\leq0$, (\ref{eq:abel conseq})
can be written in the following compact way 
\begin{equation}
  \sum_{n\leq\lambda b}\sum_{m\leq\lambda b-n}g_{2}
  \left(n\right)g_{1}\left(m\right)f\left(\frac{n+m}{\lambda}\right)=\frac{1}{\lambda}\int_{0}^{b}f^{\prime\prime}\left(w\right)\left(G_{1}*G_{2}\right)\left(\lambda w\right) \, \dx w,
  \label{eq:convpulita}
\end{equation}
where
\begin{equation}
\label{def-conv}
  \left(G_{1}*G_{2}\right)\left(x\right)
  :=
  \int_{0}^{x}G_{1}\left(x-s\right)G_{2}\left(s\right) \, \dx s
\end{equation}
is the classical convolution in the sense of Laplace. Furthermore,
it is not difficult to observe, retracing the proof of the previous
lemma, that if we assume that $f$, $f^{\prime}$, 
$f^{\prime\prime}\,:\,\left[0,+\infty\right)\rightarrow\left[0,+\infty\right)$
do not have compact support but decay at infinity sufficiently quickly,
then we get an analogous formula of (\ref{eq:abel conseq}) or (\ref{eq:convpulita}),
that is
\begin{align}
  \sum_{n>\lambda a}g_{2}(n)
  \sum_{m\geq1}g_{1}\left(m\right)f\left(\frac{n+m}{\lambda}\right) 
  & =
  G_{2}\left(\lambda a\right)\left[\int_{a}^{+\infty}G_{1}\left(\lambda v-\lambda a\right)f^{\prime}\left(v\right) \, \dx v\right]
  \label{eq:abel conseq-1}\\
 & \quad+
 \frac{1}{\lambda}\int_{a}^{+\infty}f^{\prime\prime}\left(w\right)\int_{\lambda a}^{\lambda w}G_{2}(s)G_{1}\left(\lambda w-s\right) \, \dx s \, \dx w\nonumber 
\end{align}
or, if $a\leq0$,
\begin{equation*}
\sum_{n\geq1}
\sum_{m\geq1}g_{2}
\left(n\right)g_{1}\left(m\right)f\left(\frac{n+m}{\lambda}\right)=\frac{1}{\lambda}\int_{0}^{+\infty}f^{\prime\prime}\left(w\right)\left(G_{1}*G_{2}\right)\left(\lambda w\right) \, \dx w.
\end{equation*}
Clearly, the hypothesis ``decay at infinity sufficiently quickly''
depends on the growth of $G_{j}\left(w\right),\,j=1,2$. For example,
if there exists an $\alpha>0$ such that $G_{j}\left(w\right)\ll w^{\alpha}$
as $w\rightarrow+\infty$ a possible choice of $f$ is
\[
f\left(w\right)
=
e^{-w}h\left(w\right)
\]
where $h\left(w\right)$ is $C^2$ and $h$, $h'$, $h''$ are bounded functions.
\end{rem}

We have proved the following statement: if you have sufficient information
on the convolution of averages of two arithmetical functions, you can
obtain information about the double sum/series of the same arithmetical
functions. It is not difficult to generalise the previous results to
an arbitrary fixed number of summands. Indeed, we can prove the following

\begin{cor}
Let $f:\mathbb{R}\rightarrow\mathbb{C}$ and $d\in\mathbb{N}_{>1}$
and assume that

- $f$ has compact support on $\left[a,b\right],\,a\leq0$ and $a<b$;

- $f\in C^{d-1}\left(a,b\right)$;

- $f^{\left(d-1\right)}$ is absolutely continuous on $(a,b)$;

- $f\left(a^{+}\right)$, $f^{\prime}\left(a^{+}\right),\dots,f^{\left(d-1\right)}\left(a^{+}\right)$
exist and are finite and $f\left(b^{-}\right)=f^{\prime}\left(b^{-}\right)=\dots=f^{\left(d-1\right)}\left(b^{-}\right)=0$.

Furthermore let $g_{1},\dots,g_{d}$ be arithmetic functions and $\lambda>0.$
Then, the following summation formula holds
\begin{align}
\sum_{n_{1}
\leq
\lambda b}\sum_{n_{2}\leq\lambda b-n_{1}}
&
\cdots\sum_{n_{d}\leq\lambda b-n_{1}-\dots-n_{d-1}}g_{1}(n_{1})\cdots g_{d}\left(n_{d}\right)f\left(\frac{n_{1}+\dots+n_{d}}{\lambda}\right)\label{eq:abel conseq-1-1}
\\
& =
\frac{\left(-1\right)^{d}}{\lambda^{d-1}}
\int_{0}^{b}f^{\left(d\right)}\left(w\right)\left(G_{1}*\left(G_{2}*\dots*\left(G_{d-1}*G_{d}\right)\right)\right)\left(\lambda w\right) \, \dx w.\nonumber
\end{align}
\end{cor}

\begin{proof}
We prove this formula using induction. For $d=2$ is the previous
Proposition, so assume that formula (\ref{eq:abel conseq-1-1}) holds
for $d-1$, $d\geq3$. Then, take $n_{1}\in\mathbb{N}$ and $\lambda>0$
and define $f_{n_{1},\lambda}\left(x\right):=f\left(\frac{n_{1}}{\lambda}+x\right):\left[-\frac{n_{1}}{\lambda},b-\frac{n_{1}}{\lambda}\right]\rightarrow\mathbb{R}_{0}^{+}$.
Clearly, for every $n_{1}\in\mathbb{N},$ we have that $f_{n_{1},\lambda}$
verifies the hypotheses of Proposition \ref{prop:abel sum} and since
$a=-\frac{n_{1}}{\lambda}<0,$ we have from the inductive hypothesis
that
\begin{align*}
\sum_{n_{2}
\leq
\lambda b-n_{1}}&
\cdots\sum_{n_{d}\leq\lambda b-n_{1}-\dots-n_{d-1}}g_{2}\left(n_{2}\right)\cdots g_{d}\left(n_{d}\right)f\left(\frac{n_{1}+n_{2}+\dots+n_{d}}{\lambda}\right)
\\
&=
\sum_{n_{2}
\leq
\lambda b-n_{1}}\cdots\sum_{n_{d}\leq\lambda b-n_{1}-\dots-n_{d-1}}g_{2}\left(n_{2}\right)\cdots g_{d}\left(n_{d}\right)f_{n_{1},\lambda}\left(\frac{n_{2}+\dots+n_{d}}{\lambda}\right)
\\
&=
\frac{\left(-1\right)^{d-1}}{\lambda^{d-2}}
\int_{0}^{b-\frac{n_{1}}{\lambda}}f_{n_{1},\lambda}^{\left(d-1\right)}\left(w\right)\left(G_{2}*\left(G_{3}*\dots*\left(G_{d-1}*G_{d}\right)\right)\right)\left(\lambda w\right) \, \dx w
\\
&=
\frac{\left(-1\right)^{d-1}}{\lambda^{d-2}}
\int_{0}^{b-\frac{n_{1}}{\lambda}}f^{\left(d-1\right)}\left(\frac{n_{1}}{\lambda}+w\right)\left(G_{2}*\left(G_{3}*\dots*\left(G_{d-1}*G_{d}\right)\right)\right)\left(\lambda w\right) \, \dx w.
\end{align*}
So, if we multiply both sides by $g_{1}\left(n_{1}\right)$ and summing over all $n_1\le \lambda b$, we get
\begin{align*}
\sum_{n_{1}
\leq
\lambda b}&
\sum_{n_{2}\leq\lambda b-n_{1}}\cdots\sum_{n_{d}\leq\lambda b-n_{1}-\dots-n_{d-1}}g_{1}\left(n_{1}\right)g_{2}\left(n_{2}\right)\cdots g_{d}\left(n_{d}\right)f\left(\frac{n_{1}+n_{2}+\dots+n_{d}}{\lambda}\right)
\\
&=
\frac{\left(-1\right)^{d-1}}{\lambda^{d-2}}\sum_{n_{1}\leq\lambda b}g_{1}\left(n_{1}\right)
\int_{0}^{b-\frac{n_{1}}{\lambda}}f^{\left(d-1\right)}\left(w+\frac{n_{1}}{\lambda}\right)\left(G_{2}*\left(G_{3}*\dots*\left(G_{d-1}*G_{d}\right)\right)\right)\left(\lambda w\right) \, \dx w,
\end{align*}
and now it is enough to use the same technique of the proof
of Proposition \ref{prop:abel sum}. Indeed we get
\begin{align}
&\frac{\left(-1\right)^{d-1}}{\lambda^{d-2}}
\sum_{n_{1}\leq\lambda b}g_{1}\left(n_{1}\right)
\int_{0}^{b-\frac{n_{1}}{\lambda}}f^{\left(d-1\right)}\left(w+\frac{n_{1}}{\lambda}\right)\left(G_{2}*\left(G_{3}*\dots*\left(G_{d-1}*G_{d}\right)\right)\right)\left(\lambda w\right) \, \dx w \nonumber
\\
&=
\frac{\left(-1\right)^{d}}{\lambda^{d-2}}
\int_{0}^{\lambda b}G_{1}\left(s\right)\frac{\dx}{\dx s}\left(\int_{0}^{b-\frac{s}{\lambda}}f^{\left(d-1\right)}\left(w+\frac{s}{\lambda}\right)\left(G_{2}*\left(G_{3}*\dots*\left(G_{d-1}*G_{d}\right)\right)\right)\left(\lambda w\right) \, \dx w\right) \, \dx s \nonumber
\\
&=
\frac{\left(-1\right)^{d}}{\lambda^{d-1}}
\int_{0}^{\lambda b}G_{1}\left(s\right)\int_{\frac{s}{\lambda}}^{b}f^{\left(d\right)}\left(v\right)\left(G_{2}*\left(G_{3}*\dots*\left(G_{d-1}*G_{d}\right)\right)\right)\left(\lambda v-s\right) \, \dx v \, \dx s.\label{eq:corollario conv1}
\end{align}
Now, since 
\begin{equation}
\left(G_{2}*\left(G_{3}*\dots*\left(G_{d-1}*G_{d}\right)\right)\right)\left(\lambda v-s\right)\equiv0\label{eq:conv eq to 0}
\end{equation}
if $v<\frac{s}{\lambda}$ we can extend the domain of the inner integral
from $\left[\frac{s}{\lambda},b\right]$ to $\left[0,b\right]$ and
then apply the Fubini theorem, so we can write (\ref{eq:corollario conv1})
as
\begin{equation}
\frac{\left(-1\right)^{d}}{\lambda^{d-1}}
\int_{0}^{b}f^{\left(d\right)}\left(v\right)\int_{0}^{\lambda b}G_{1}\left(s\right)\left(G_{2}*\left(G_{3}*\dots*\left(G_{d-1}*G_{d}\right)\right)\right)\left(\lambda v-s\right) \, \dx s \, \dx v.\label{eq:generaliz first prop1}
\end{equation}
Now the condition for (\ref{eq:conv eq to 0}), that is, $v<\frac{s}{\lambda}$,
obviously is equivalent to $s>\lambda v$. Hence we can reduce the
domain of the inner integral of (\ref{eq:generaliz first prop1})
from $\left[0,\lambda b\right]$ to $\left[0,\lambda v\right]$, hence
we get that (\ref{eq:generaliz first prop1}) is equal to
\[
\frac{\left(-1\right)^{d}}{\lambda^{d-1}}
\int_{0}^{b}f^{\left(d\right)}\left(v\right)\int_{0}^{\lambda v}G_{1}\left(s\right)\left(G_{2}*\left(G_{3}*\dots*\left(G_{d-1}*G_{d}\right)\right)\right)\left(\lambda v-s\right) \, \dx s \, \dx v
\]
\[
=
\frac{\left(-1\right)^{d}}{\lambda^{d-1}}
\int_{0}^{b}f^{\left(d\right)}\left(v\right)\left(G_{1}*\left(G_{2}*\left(G_{3}*\dots*\left(G_{d-1}*G_{d}\right)\right)\right)\right)\left(\lambda v\right) \, \dx s \, \dx v
\]
as wanted.
\end{proof}
As mentioned before, the convolution of summatory functions encompasses,
essentially, the information of the weighted averages. In the next
Lemma, we show that the convolution of $d$ summatory functions with
$d\geq2$ is the Ces\`aro average of order $k=d-1$ of the function
that counts the representations of an integer as a sum of elements defined
by the $d$ arithmetical functions.
\begin{lem}
Let $d\geq2$ be a natural number, $g_{i}:\mathbb{N}\rightarrow\mathbb{C},i=1,\dots,d$ arithmetical functions, $\text{g}:=\left(g_{1},\dots,g_{d}\right)$
and $x\in\mathbb{R}_{0}^{+}.$ Consider
\[
G_{i}(x):=\sum_{n\leq x}g_{i}(n),
\qquad
\mathcal{G}\left(n\right)=\mathcal{G}\left(n,\text{g}\right)
:=
\sum_{m_{1}+\dots+m_{d}=n}g_{1}\left(m_{1}\right) \cdots g_{d}\left(m_{d}\right).
\]
Then, for $x\geq0$, we have
\begin{equation}
\frac{1}{\left(d-1\right)!}
\sum_{n\leq x}\mathcal{G}
\left(n\right)\left(x-n\right)^{d-1}=\left(G_{1}*\left(G_{2}*\dots*\left(G_{d-1}*G_{d}\right)\right)\right)\left(x\right).\label{eq: conv psi gen}
\end{equation}
\end{lem}

\begin{proof}
We prove this Lemma by induction. We start with $d=2$ and we consider
the operators $\mathcal{L}_{x}=\mathcal{L}_{x}\left(\text{g}\right)$
and $\mathcal{M}_{x}=\mathcal{M}_{x}\left(\text{g}\right)$ which
associate to $x$ the first and second members of (\ref{eq: conv psi gen}),
respectively. It is easy to see that $\mathcal{L}_{x}$ and $\mathcal{M}_{x}$
are bilinear operators, so, for proving (\ref{eq: conv psi gen}),
it is enough to show it for the following basis
\[
\mathcal{B}
=
\left\{ \left(\Delta_{N_{1}},\Delta_{N_{2}}\right):N_{1},N_{2}\in\mathbb{N}\right\} 
\]
where $\Delta_{N}=\Delta_{N}(n):=\delta_{n,N}$ and $\delta_{i,j}$ is Kronecker's delta function. Then, if $\left(\Delta_{N_{1}},\Delta_{N_{2}}\right)\in\mathcal{B}$,
we get
\[
\mathcal{G}\left(n\right)=\begin{cases}
1, & \text{if } n=N_{1}+N_{2}\\
0, & \text{otherwise,}
\end{cases}
\]
that is, $\mathcal{G}\left(n\right)=\Delta_{N_{1}+N_{2}}$. Furthermore
\[
G_{i}\left(x\right)=\begin{cases}
0, & \text{if } x<N_{i}\\
1, & \text{if } x\geq N_{i}
\end{cases}
\]
for $i=1,2$. Then 
\[
\mathcal{L}_{x}\left(\Delta_{N_{1}},\Delta_{N_{2}}\right)=\begin{cases}
0, & \text{if } x<N_{1}+N_{2}\\
x-\left(N_{1}+N_{2}\right), & \text{if } x\geq N_{1}+N_{2}.
\end{cases}
\]
In addition, recalling the definition of $G_1 * G_2$ in \eqref{def-conv}, we note that
\[
\mathcal{M}_{x}\left(\text{g}\right)=\begin{cases}
0, & \text{if } x<N_{1}+N_{2}\\
x-\left(N_{1}+N_{2}\right), & \text{if } x\geq N_{1}+N_{2}
\end{cases}
\]
and this the base case. Now, assume $d\geq3$ and, again consider
the operators $\mathcal{L}_{x}=\mathcal{L}_{x}\left(\text{g}\right)$
and $\mathcal{M}_{x}=\mathcal{M}_{x}\left(\text{g}\right)$ which
associate to $x$ the first and second members of (\ref{eq: conv psi gen}),
respectively. In this case, we can observe that $\mathcal{L}_{x},\,\mathcal{M}_{x}$
are multilinear operators. Hence it suffices to verify (\ref{eq: conv psi gen})
on a basis of the complex vector space of $d$-tuples of complex
sequences. As in the previous case, we choose
\[
\mathcal{B}
=
\left\{ \left(\Delta_{N_{1}},\dots,\Delta_{N_{d}}\right):N_{1},\dots,N_{d}\in\mathbb{N}\right\}. 
\]
For the sake of brevity,
we write $N=N_{1}+\dots+N_{d}$ throughout the proof. We remark that
if $\left(\Delta_{N_{1}},\dots,\Delta_{N_{d}}\right)\in\mathcal{B}$,
then
\[
\mathcal{G}\left(n\right)=\begin{cases}
1, & \text{if } n=N\\
0, & \text{otherwise,}
\end{cases}
\]
that is, $\mathcal{G}\left(n\right)=\Delta_{N}.$ Hence
\[
\mathcal{L}_{x}\left(\Delta_{N_{1}},\dots,\Delta_{N_{d}}\right)=\begin{cases}
0, & \text{if } x<N\\
\frac{\left(x-N\right)^{d-1}}{\left(d-1\right)!}, & \text{if } x\geq N.
\end{cases}
\]
We now turn to $\mathcal{M}_{x}$. Let $M=N-N_{1}.$ For $d\geq 3$ we assume inductively
that 
\[
\left(G_{2}*\left(G_{3}*\dots*\left(G_{d-1}*G_{d}\right)\right)\right)\left(x\right)=\begin{cases}
0, & \text{if } x<M\\
\frac{\left(x-M\right)^{d-2}}{\left(d-2\right)!}, & \text{if } x\geq M.
\end{cases}
\]
Now, note that $\mathcal{M}_{x}\left(\Delta_{N_{1}},\dots,\Delta_{N_{d}}\right)=0$
if $x<N$ because in the integral
\[
\int_{0}^{x}G_{1}\left(x-s\right)\left(G_{2}*\left(G_{3}*\dots*\left(G_{d-1}*G_{d}\right)\right)\right)\left(s\right) \, \dx s
\]
either $G_{1}\left(x-s\right)=0$ or $\left(G_{2}*\left(G_{3}*\dots*\left(G_{d-1}*G_{d}\right)\right)\right)\left(s\right)=0$.
If $x\geq N$, then
\begin{align*}
\int_{0}^{x}G_{1}\left(x-s\right)\left(G_{2}*\left(G_{3}*\dots*\left(G_{d-1}*G_{d}\right)\right)\right)\left(s\right) \, \dx s 
& =
\int_{M}^{x-N_{1}}\frac{\left(s-M\right)^{d-2}}{\left(d-2\right)!} \, \dx s\\
 & =
 \frac{\left(x-N\right)^{d-1}}{\left(d-1\right)!},
\end{align*}
and thus $\mathcal{L}_{x},\,\mathcal{M}_{x}$ agree on the whole space.
The same results may be obtained, in a less elementary fashion, using
the Laplace convolution and inversion theorem.
\end{proof}

\section{the case \texorpdfstring{$g_{1}(n)=g_{2}(n)=\Lambda(n)$}{g1g2 = Lambda}}

In this section we will analyse the important case
$g_{1}(n)=g_{2}(n)=\Lambda(n)$. We will show that using well-known
facts about the Ces\`aro average of Goldbach numbers with $k=1$, which matches 
the convolution $(\psi*\psi)(x)$ and the results of the previous section, we are able to
provide new general formulas for weighted averages of Goldbach's numbers.
Regarding these topics,we recall the results in Languasco \& Zaccagnini \cite{LANZAC2}
and Br\"udern, Kaczorowski \& Perelli \cite{BKP}; In particular, some results 
of the last one we will be useful for our work.
We let
\begin{equation}
\label{def-RG}
  R_G(n)
  :=
  \sum_{m_{1}+m_{2}=n}\Lambda\left(m_{1}\right)\Lambda\left(m_{2}\right)
  \qquad\text{and}\qquad
  C_k(x)
  :=
  \frac1{\Gamma(k + 1)}
  \sum_{n \le x} (x - n)^k R_G(n),
\end{equation}
where $k \in \mathbb{R}^+$.
We will show that our formula and a sufficient knowledge of the
behaviour of the convolution of the function $\psi(x)$
with itself allow to reprove some well-known facts about the function
$R_G(n)$
in a very easy and natural way. Indeed, in this case formula (\ref{eq:convpulita})
becomes
\begin{equation}
  \Sigma(f; \lambda; [0,b])
  :=
  \sum_{n\leq\lambda b}\sum_{m\leq\lambda b-n}\Lambda\left(n\right)\Lambda\left(m\right)f\left(\frac{n+m}{\lambda}\right)
  =
  \frac{1}{\lambda}\int_{0}^{b}f^{\prime\prime}\left(w\right)\left(\psi*\psi\right)\left(\lambda w\right) \, \dx w, \label{eq:star formula}
\end{equation}
where $f$ is a function that satisfies the hypotheses of Proposition
\ref{prop:abel sum} (Remark \ref{rem: variante}) with $a=0$ and
$b\in\mathbb{R}$ ($b=+\infty)$.

It is convenient to introduce some more notation at this point.
Following \cite{BKP} we write
\begin{equation}
\label{def-Z}
  Z_{\lambda}\left(w\right)
  :=
  \sum_{\rho}
    \frac{\lambda^{\rho}\Gamma\left(\rho\right)}{\Gamma\left(\rho+w+1\right)},
\end{equation}
where the summation is over the non-trivial zeros of the Riemann zeta-function.
We also write
\begin{align*}
  M_G(x)
  &:=
  \frac{x^{3}}{6}-2\sum_{\rho}\frac{x^{\rho+2}}{\rho\left(\rho+1\right)\left(\rho+2\right)}
  +\sum_{\rho_{1}}\sum_{\rho_{2}}\frac{x^{\rho_{1}+\rho_{2}+1}\Gamma\left(\rho_{1}\right)\Gamma\left(\rho_{2}\right)}{\Gamma\left(\rho_{1}+\rho_{2}+2\right)} \\
\nonumber
  &\ =
  \frac{x^3}6
  -
  2 x^2 Z_x(2)
  +
  \sum_{\rho} x^{\rho + 1} \Gamma(\rho) Z_x(\rho + 1)
\end{align*}
for the main term and the first two secondary terms in the Ces\`aro
average for Goldbach numbers.
For a possibly infinite interval $I$ of the real line we set
\begin{align}
\label{def-M0}
  \M_0(f; \lambda; I)
  &:=
  \lambda^2 \int_I w f(w) \, \dx w, \\
\label{def-M1}
  \M_1(f; \lambda; I)
  &:=
  \sum_{\rho} \frac{\lambda^{\rho+1}}{\rho}
    \int_I w^{\rho} f(w) \, \dx w, \\
\label{def-M2}
  \M_2(f; \lambda; I)
  &:=
  \sum_{\rho_1}
    \sum_{\rho_2}
      \frac{\lambda^{\rho_{1}+\rho_{2}} \Gamma(\rho_1) \Gamma(\rho_2)}
           {\Gamma\left(\rho_{1}+\rho_{2}\right)}
      \int_I w^{\rho_{1}+\rho_{2}-1} f(w) \, \dx w.
\end{align}
If $f$ satisfies suitable conditions, we will also take $b = +\infty$.
  
\begin{cor}
\label{convolution}(Explicit formula convolution) Let $x>4$. Then
\begin{equation}
\label{eq:conv}
  \left(\psi*\psi\right)\left(x\right)
  =
  M_G(x) + E(x),
\end{equation}
where $E(x)$ is a function that could be made explicit and $E\left(x\right)=O\left(x^{2}\right)$
as $x\rightarrow+\infty.$
\end{cor}

\begin{proof}
It is a straightforward consequence of the previous lemma and the
fact that $C_1(x)$ is the Ces\`aro average of the Goldbach's numbers with $k=1$.
It is well known that, in this case, we have
\[
  C_1(x)
  =
  \sum_{n\leq x}R_{G}\left(n\right)\left(x-n\right)
  =
  M_G(x)
  +
  E(x),
\]
where $E(x)=O\left(x^{2}\right)$ as $x\rightarrow+\infty$, and the
series that runs over the non trivial zeros of the Riemann zeta-function
are absolutely convergent; see \cite{BKP} for a complete proof of
these claims in the general case with Ces\`aro weight of order $k>0$.
\end{proof}

For complex $\alpha$ and an interval $I$ as above we write
\begin{equation}
\label{def-I}
  \I(\alpha; I) := \int_I w^\alpha f''(w) \, \dx w.
\end{equation}
From the previous results we can establish the following general formula:
\begin{thm}
\label{thm:main}
Let $\lambda>4$ and let $f:\mathbb{R}_{0}^{+}\rightarrow\mathbb{C}$
be a function that satisfies the hypotheses of Proposition \ref{prop:abel sum}
with $a=0$. Then, we have 
\begin{equation}
\label{eq:main}
  \Sigma(f; \lambda; [0,b])
  =
  \M_0(f; \lambda; [0,b]) - 2 \M_1(f; \lambda; [0,b]) + \M_2(f; \lambda; [0,b]) \\
  +
  O\left(\lambda\int_{0}^{b}w^{2}\left|f^{\prime\prime}\left(w\right)\right| \, \dx w\right).
\end{equation}
Furthermore, we can write an explicit version of the previous formula
\begin{equation}
\label{eq:main-3}
  \Sigma(f; \lambda; [0,b])
  =
  \M +
  \frac{1}{\lambda}\int_{4/\lambda}^{b}E(\lambda w)f^{\prime\prime}\left(w\right) \, \dx w 
\end{equation}
where $I = [4/\lambda, b]$, $E(x)$ is the error term in
(\ref{eq:conv}) and
\[
  \M
  :=
  \frac{\lambda^{2}}{6} \I(3; I)
  -
  2\sum_{\rho}\frac{\lambda^{\rho+1}}{\rho(\rho+1)\left(\rho+2\right)}
  \I(\rho + 2; I)
  +
  \sum_{\rho_{1}}\sum_{\rho_{2}}\frac{\lambda^{\rho_{1}+\rho_{2}}\Gamma\left(\rho_{1}\right)\Gamma\left(\rho_{2}\right)}{\Gamma\left(\rho_{1}+\rho_{2}+2\right)}
  \I(\rho_{1}+\rho_{2}+1; I).
\]
\end{thm}

\begin{proof}
Formula \eqref{eq:main-3} is a trivial application of formulas (\ref{eq:convpulita}),
(\ref{eq:conv}) and the absolute convergence of the series that runs
of the non-trivial zeros of the Riemann zeta-function.
We just remark that $\psi * \psi(x) = 0$ for $x \le 4$.
We focus
on the first formula \eqref{eq:main}. A direct application of Corollary \ref{convolution}
with $E(x)=O\left(x^{2}\right)$ and (\ref{eq:star formula}) gives
the formula
\[
  \Sigma(f; \lambda; [0,b])
  =
  \M + 
  O\left(\lambda\int_{4/\lambda}^{b}w^{2}\left|f^{\prime\prime}\left(w\right)\right| \, \dx w\right).
\]
Again note that, since all the series converge absolutely, it is
not difficult to prove that we can exchange the symbols.
We set $I_1 = [0, b]$ and $I_2 = [0, 4 / \lambda]$.
Now, clearly
\[
  \frac{\lambda^{2}}{6} \I(3; I)
  =
  \frac{\lambda^{2}}{6} \I(3; I_1)
  -
  \frac{\lambda^{2}}{6} \I(3; I_2)
  =
  \frac{\lambda^{2}}{6} \I(3; I_1)
  +
  O\left(\lambda\int_{0}^{4/\lambda}w^{2}\left|f^{\prime\prime}\left(w\right)\right| \, \dx w\right).
\]
Finally, it remains to observe that
\[
  \left|\sum_{\rho}\frac{\lambda^{\rho+1}}{\rho\left(\rho+1\right)\left(\rho+2\right)} \I(\rho+2; I_2) \right|
  \ll\lambda^{2}\int_{0}^{4/\lambda}w^{3}\left|f^{\prime\prime}\left(w\right)\right| \, \dx w 
  \ll\lambda\int_{0}^{4/\lambda}w^{2}\left|f^{\prime\prime}\left(w\right)\right| \, \dx w
\]
and
\begin{align*}
  &\left|\sum_{\rho_{1}}\sum_{\rho_{2}}\frac{\lambda^{\rho_{1}+\rho_{2}}\Gamma\left(\rho_{1}\right)\Gamma\left(\rho_{2}\right)}{\Gamma\left(\rho_{1}+\rho_{2}+2\right)}
  \I(\rho_{1}+\rho_{2}+1; I_2) \right|
  \\
 &\qquad
 \ll
  \lambda^{2}\int_{0}^{4/\lambda}w^{3}\left|f^{\prime\prime}\left(w\right)\right| \, \dx w 
  \ll
  \lambda\int_{0}^{4/\lambda}w^{2}\left|f^{\prime\prime}\left(w\right)\right| \, \dx w
\end{align*}
and so the claim follows using integration by parts twice for the main
terms.
\end{proof}

\begin{rem}
It is important to note that the number of terms in the asymptotic
formula (\ref{eq:main}) depends only on the number of terms in the
explicit formula of $\left(\psi*\psi\right)\left(x\right)$, that
is, in the explicit formula for $C_1(x)$.
As we have already observed, we are able to write complete explicit formula
for the Ces\`aro average of Goldbach numbers for all real orders $k>0$;
this implies that we can write (\ref{eq:main}) with more terms if
we needed, and this can be useful in some situations where error estimation
is very delicate. For example, using the last formula for $G_k (N)$ in \cite[section 1]{BKP} for $x>4$ we can write
\[
  (\psi * \psi)(x)
  =
  M_G(x)
  -
  \frac{\zeta^{\prime}}{\zeta}\left(0\right)x^{2}\\
  +
  2 \frac{\zeta^{\prime}}{\zeta}\left(0\right)
  \sum_{\rho}\frac{x^{\rho+1}}{\rho\left(\rho+1\right)}
  +O\left(x\right),
\]
and, using the arguments in the proof of Theorem \ref{thm:main},
we get
\begin{align*}
  \Sigma(f; \lambda; [0,b])
  &=
  \M_0(f; \lambda; [0,b]) - 2 \M_1(f; \lambda; [0,b]) + \M_2(f; \lambda; [0,b]) 
   -2\frac{\zeta^{\prime}}{\zeta}\left(0\right)\lambda\int_{0}^{b}f\left(w\right) \, \dx w
   \\
  &\qquad
 +2\frac{\zeta^{\prime}}{\zeta}\left(0\right)\sum_{\rho}\lambda^{\rho}\int_{0}^{b}w^{\rho-1}f\left(w\right) \, \dx w
  +O\left(\int_{0}^{b}w\left|f^{\prime\prime}\left(w\right)\right| \, \dx w\right).\nonumber 
\end{align*}
\end{rem}

Furthermore we can write a version of (\ref{eq:conv}) with an integral
representation of the remainder. Indeed, from \cite{BKP}, we know
that
\[
  C_1(x)
  =
  \sum_{n\leq x}R_{G}\left(n\right)\left(x-n\right)
  =
  \frac{1}{\left(2\pi i\right)^{2}}\int_{\left(2\right)}\int_{\left(2\right)}\frac{\zeta^{\prime}}{\zeta}\left(w\right)\frac{\zeta^{\prime}}{\zeta}\left(s\right)\frac{\Gamma\left(w\right)\Gamma\left(s\right)}{\Gamma\left(w+s+2\right)}x^{s+w+1} \, \dx s \, \dx w
\]
(see formula $(1.9)$ of \cite{BKP}). So, moving the two integrals
one after the other to the complex line with real part equal to $-1/2$,
which is possible by \cite{BKP}, we get from the residue
theorem that
\begin{align*}
  \left(\psi*\psi\right)\left(x\right)
  &=\sum_{n\leq x}R_{G}\left(n\right)\left(x-n\right)
  =
  M_G(x) - 2\frac{\zeta^{\prime}}{\zeta}\left(0\right)x^{2} \\
  &\qquad 
  +2\frac{\zeta^{\prime}}{\zeta}\left(0\right)\sum_{\rho}\frac{x^{\rho+1}}{\rho\left(\rho+1\right)}+\frac{\zeta^{\prime}}{\zeta}\left(0\right)^{2}\frac{x}{2}
  +E_1\left(x\right),\nonumber 
\end{align*}
where
\begin{align*}
E_1\left(x\right) & =\frac{1}{2\pi i}\int_{\left(-\frac{1}{2}\right)}\frac{\zeta^{\prime}}{\zeta}\left(s\right)\frac{x^{s+2}}{\left(s+2\right)\left(s+1\right)s} \, \dx s-\frac{1}{2\pi i}\sum_{\rho}\Gamma\left(\rho\right)
\int_{\left(-\frac{1}{2}\right)}\frac{\zeta^{\prime}}{\zeta}\left(s\right)\frac{\Gamma\left(s\right)}{\Gamma\left(\rho+s+2\right)}x^{s+\rho+1} \, \dx s
 \\
  &\qquad
  -\frac{\zeta^{\prime}}{\zeta}\left(0\right)\frac{1}{2\pi i}\int_{\left(-\frac{1}{2}\right)}\frac{\zeta^{\prime}}{\zeta}\left(s\right)\frac{x^{s+1}}{\left(s+1\right)s} \, \dx s \\
&\qquad
+\frac{1}{\left(2\pi i\right)^{2}}\int_{\left(-\frac{1}{2}\right)}\int_{\left(-\frac{1}{2}\right)}\frac{\zeta^{\prime}}{\zeta}\left(w\right)\frac{\zeta^{\prime}}{\zeta}\left(s\right)\frac{\Gamma\left(w\right)\Gamma\left(s\right)}{\Gamma\left(w+s+2\right)}x^{s+w+1} \, \dx s \, \dx w.\nonumber 
\end{align*}

\begin{rem}
If we assume that $f\left(w\right)$ does not have compact support but it
decays at infinity sufficiently fast (for example there exists an
$\alpha>0$ such that $f^{\left(j\right)}\left(w\right)\ll e^{-\alpha w},\,j=0,1,2$)
then it is not difficult to see that Theorem \ref{thm:main} continues
to hold in the form
\begin{align}
\sum_{n\geq1}\sum_{m\geq1}\Lambda\left(n\right)\Lambda\left(m\right)f\left(\frac{n+m}{\lambda}\right) & =\lambda^{2}\int_{0}^{+\infty}wf\left(w\right) \, \dx w-2\sum_{\rho}\frac{\lambda^{\rho+1}}{\rho}\int_{0}^{+\infty}w^{\rho}f\left(w\right) \, \dx w\label{eq:main-2}\\
 &\quad +
 \sum_{\rho_{1}}\sum_{\rho_{2}}\frac{\lambda^{\rho_{1}+\rho_{2}}\Gamma\left(\rho_{1}\right)\Gamma\left(\rho_{2}\right)}{\Gamma\left(\rho_{1}+\rho_{2}\right)}\int_{0}^{+\infty}w^{\rho_{1}+\rho_{2}-1}f\left(w\right) \, \dx w\nonumber \\
 &\quad +
 O\left(\lambda\int_{0}^{+\infty}w^{2}\left|f^{\prime\prime}\left(w\right)\right| \, \dx w\right).\nonumber 
\end{align}
\end{rem}

\section{the Ces\`aro weight case for the Goldbach numbers}

In this part we show an application of the previous theorems. We
have already seen that the Ces\`aro weight of order $d$ is strictly
linked to the number of addends in an additive problem. Clearly, if
$d$ is sufficiently large, then the Ces\`aro weight of order $d-1$
is smooth, so it is an obvious choice for smoothing an arbitrary average,
not only averages for counting functions with $d$ terms. So it is
natural to study, for example, a binary problem, like the Goldbach
problem, with a general Ces\`aro weight of order $k>0$. In this part,
we show that the well-known results about the Ces\`aro average of Goldbach
numbers developed in a series of papers (see \cite{BKP,CAN3,GOLYAN,LANZAC2})
can be obtained very easily from the ``important order'' case $k=1.$
Essentially, from this particular case we are able to obtain every
case with $k>0.$

An important remark is that we indeed show the same
asymptotic formula present in \cite{BKP}, but using
our different approach. However, we will need a fundamental estimation
present in \cite{BKP} regarding a double series
that runs over the non-trivial zeros of the Riemann Zeta function.
For $k\in\mathbb{C}$ with $\Re(k) > 1$ we let
\[
  \phi_k(x)
  :=
  \begin{cases}
    \dfrac{\left(1-x\right)^{k}}{\Gamma\left(k+1\right)}
    &\text{for $0 < x < 1$,} \\
    0 &\text{otherwise.}
  \end{cases}
\]
Then, it is easy to see
that $\phi_k$ satisfies the hypotheses of Proposition \ref{prop:abel sum}.
We define the function $\E\left(\lambda,k\right)$ as follows:
\begin{equation}
\label{def-E-gotico}
  \E\left(\lambda,k\right)
  =
  \Sigma(\phi_k; \lambda; [0,1]) - \M_0(\phi_k; \lambda; [0,1])
  + 2 \M_1(\phi_k; \lambda; [0,1]) - \M_2(\phi_k, \lambda; [0,1]),
\end{equation}
where $\lambda>4$ is a fixed real number.
We can compute exactly the values of $\M_0$, $\M_1$ and $\M_2$ using
the definition of $\phi_k$ and recalling definitions \eqref{def-M0},
\eqref{def-M1} and \eqref{def-M2}, we see that
\begin{align}
\label{value-M0}
  \M_0(\phi_k; \lambda; [0,1])
  &=
  \frac{\lambda^{2}}{\Gamma\left(k+3\right)}, \\
\label{value-M1}
  \M_1(\phi_k; \lambda; [0,1])
  &=
  \sum_{\rho}\frac{\lambda^{\rho+1}\Gamma(\rho)}{\Gamma(\rho+k+2)}, \\
\label{value-M2}
  \M_2(\phi_k; \lambda; [0,1])
  &=
  \sum_{\rho_{1}}\Gamma\left(\rho_{1}\right)\lambda^{\rho_{1}}
  \sum_{\rho_{2}}\frac{\lambda^{\rho_{2}}\Gamma\left(\rho_{2}\right)}
      {\Gamma\left(\rho_{1}+\rho_{2}+k+1\right)}.
\end{align}

Let us analyse the series in \eqref{value-M1} and \eqref{value-M2}.
We already know that these series converge absolutely for real $k>0$
but we want to show that the above actually holds for $k\in\mathbb{C},\,\text{Re}(k)>0$.
By Stirling's formula
\[
\left|\Gamma\left(x+iy\right)\right|\sim\sqrt{2\pi}e^{-\pi\left|y\right|/2}\left|y\right|^{x-1/2}
\]
as $\left|y\right|\rightarrow+\infty$, which holds uniformly for $x\in\left[x_{1},x_{2}\right]$,
$x_{1},x_{2}\in\mathbb{R}$ fixed, we can conclude that $\sum_{\rho}\frac{\lambda^{\rho+1}\Gamma\left(\rho\right)}{\Gamma\left(\rho+k+2\right)}$
converges absolutely for every fixed $k\in\mathbb{C}$ with $\text{Re}(k)>0$
and uniformly for every compact region to the right of $\text{Re}(k) = 0$.

The second series is more delicate to analyse.

We need definition \eqref{def-Z}, and recall the following result which
is very important for our aims.
It is Proposition 2 from \cite{BKP}.

\begin{prop}
\label{prop:perelli}Let $\lambda\geq4$ be a natural number. Then
\textup{$Z_{\lambda}\left(w\right)$} extends to an entire function.
Moreover, there is a real number $C$ such that for any $\delta$
with $0<\delta<1$ and $\left|w+m\right|>\delta$ for all integers
$m\geq1$ we have
\[
\left|Z_{\lambda}\left(w\right)\right|\leq\frac{C}{\delta\left|\Gamma\left(w+1\right)\right|}
\left(\lambda^{\left|u\right|+1}+2^{\left|u\right|}\log\left(\left|w\right|+2\right)\right)
\]
for any $u \in \mathbb{R}$, where $u=\text{Re}(w)$. In addition, if we have $u \leq -3/2$, then
\[
\left|Z_{\lambda}\left(w\right)\right|\leq\frac{C}{\delta\left|\Gamma\left(w+1\right)\right|}
\left(\lambda^{\left|u\right|}\log\left(\lambda\right)+2^{\left|u\right|}\log\left(\left|w\right|\right)\right).
\]
\end{prop}

In our case we are considering
\[
\sum_{\rho_{1}}\Gamma\left(\rho_{1}\right)\lambda^{\rho_{1}}\sum_{\rho_{2}}\frac{\lambda^{\rho_{2}}\Gamma\left(\rho_{2}\right)}{\Gamma\left(\rho_{1}+\rho_{2}+k+1\right)}=\sum_{\rho}\Gamma\left(\rho\right)\lambda^{\rho}Z_{\lambda}\left(\rho+k\right),
\]
hence
\[
\sum_{\rho_{1}}
\left|\Gamma\left(\rho_{1}\right)\lambda^{\rho_{1}}
\sum_{\rho_{2}}\frac{\lambda^{\rho_{2}}\Gamma\left(\rho_{2}\right)}{\Gamma\left(\rho_{1}+\rho_{2}+k+1\right)}\right|
\leq\frac{C}{\delta}
\sum_{\rho}\frac{\left|\Gamma\left(\rho\right)\right|(\lambda^{\left|\beta
+\text{Re}(k)\right|+1}+2^{\left|\beta+\text{Re}(k)\right|}\log\left(\left|\rho+k\right|+2\right))}{\left|\Gamma\left(\rho+k+1\right)\right|}
\]
where $\beta=\text{Re}(\rho)$ and the last sum converges absolutely for $\text{Re}(k)>0$, by Stirling's approximation.

So we can conclude every term in equation \eqref{def-E-gotico}
is defined and convergent for $\text{Re}(k)>0$
and uniformly convergent in any compact region to the right of $\text{Re}(k)=0$
and so $\E\left(\lambda,k\right)$ is an analytic function in $k$
and \eqref{value-M1} and \eqref{value-M2} give the analytic
continuation of $\E\left(\lambda,k\right)$ to $\text{Re}(k)>0$. 

Now, in the case of $k\in\mathbb{R},\,k>1$, we have, since we can
apply the estimation in Theorem \ref{thm:main}, that
\begin{equation}
  \E\left(\lambda,k\right)
  =
  O\left(\frac{\lambda}{\Gamma\left(k-1\right)}\int_{0}^{1}w^{2}\left(1-w\right)^{k-2} \, \dx w\right)=O\left(\frac{\lambda}{\Gamma\left(k+2\right)}\right)\label{eq:stima E}
\end{equation}
and clearly the right part of (\ref{eq:stima E}) is defined and convergent
for $\text{Re}(k)>0$ and the implicit constant does not depend on
$k$. Hence, recalling \eqref{def-RG}, we have proved the following theorem:
\begin{thm}
Let $k\in\mathbb{R}$ with $k>0$ and $\lambda>4$ be a natural number. Then
\begin{align*}
  \frac{1}{\Gamma\left(k+1\right)}
  \sum_{n<\lambda} R_G\left(n\right)\left(1-\frac{n}{\lambda}\right)^{k}
  =&
  \M_0(\phi_k; \lambda; [0,1]) -2 \M_1(\phi_k; \lambda; [0,1])
\\  
  &+
  \M_2(\phi_k; \lambda; [0,1])
  +O\left(\frac{\lambda}{\Gamma\left(k+2\right)}\right).
\end{align*}
\end{thm}

According to definition \eqref{def-RG}, the left-hand side above is
simply $\lambda^{-k} C_k(\lambda)$.

\begin{rem*}
It is interesting to note that every time we will
have to deal with a convolution of two arithmetic functions whose
explicit formulas have some series that runs over the non-trivial
zeros of $\zeta(s)$, then some double series that runs over the non-trivial
zeros, and the problem of its convergence, will arise. Hence, the
ideas developed in Proposition \ref{prop:perelli} deserves to be
investigated and generalised since they have an important role also
with our approach.
\end{rem*}

\section{power series and Dirichlet series}

In this section we show that using the previous results we can give a
simple alternative proof of the explicit formula for the power series
\[
\sum_{n\geq1}R_G\left(n\right)e^{-ny},\,y>0
\]
when $y$ is a real number.
Indeed, take $g_{1}\left(n\right)=g_{2}\left(n\right)=\Lambda\left(n\right),\,\lambda=N>0$
and $f\left(w\right)=f_{N,y}\left(w\right)=e^{-wNy},\,y>0.$ Then
we have
\[
  \sum_{n\geq1}
    \sum_{m\geq1}
      \Lambda\left(n\right)\Lambda\left(m\right)f\left(\frac{n+m}{N}\right)
  =
  \Bigl( \sum_{n\geq1}\Lambda\left(n\right)e^{-ny} \Bigr)^2
  =
  \sum_{n\geq1}R_G\left(n\right)e^{-ny}.
\]
So, applying (\ref{eq:main-2}), we get
\begin{align*}
\sum_{n\geq1}R_G\left(n\right)e^{-ny} & 
=N^{2}\int_{0}^{+\infty}we^{-wNy} \, \dx w-2\sum_{\rho}\frac{N^{\rho+1}}{\rho}\int_{0}^{+\infty}w^{\rho}e^{-wNy} \, \dx w\\
 &\quad
  +\sum_{\rho_{1}}\sum_{\rho_{2}}\frac{N^{\rho_{1}+\rho_{2}}\Gamma\left(\rho_{1}\right)\Gamma\left(\rho_{2}\right)}{\Gamma\left(\rho_{1}+\rho_{2}\right)}\int_{0}^{+\infty}w^{\rho_{1}+\rho_{2}-1}e^{-wNy} \, \dx w\\
 & \quad
 +O\left(N^{3}y^{2}\int_{0}^{+\infty}w^{2}e^{-wNy} \, \dx w\right)
\end{align*}
and simple computations lead to
\[
  \sum_{n\geq1}R_G\left(n\right)e^{-ny}
  =
  \frac{1}{y^{2}}-2\sum_{\rho}y^{-\rho-1}\Gamma\left(\rho\right)
  +
  \Bigl( \sum_{\rho}y^{-\rho}\Gamma\left(\rho\right) \Bigr)^2
  +O\left(\frac{1}{y}\right).
\]
Again, we observe that, if needed, we can add more terms to our
formula. 

We show another well-known fact in a very easy way, that is, assuming
RH, the Dirichlet series $\Phi(s)=\sum_{n\geq1}\frac{R_G\left(n\right)}{n^{s}}$
admits a meromorphic continuation to the half plane $\text{Re}(s)>1$
with poles at $s=2$ and $s=1+\rho$; this was proved in \cite{EGMA}
and the demonstration of these facts occupied an important part of the whole article.
We will show that our formula, in just a few lines, reproves these claims very easily. Indeed,
fix $\lambda>0$ and $s\in\mathbb{C},\text{Re}(s)>2$ and consider
the function $f:\left[\frac{1}{\lambda},+\infty\right)\rightarrow\mathbb{C},\,f\left(w\right):=w^{-s}.$
Then we can conclude that
\[
\sum_{n\geq1}\frac{R_G\left(n\right)}{n^{s}}=\sum_{n\geq1}
\sum_{m\geq1}\frac{\Lambda\left(n\right)\Lambda\left(m\right)}{\left(m+n\right)^{s}}
\]
converges. Now, if we apply formula (\ref{eq:abel conseq-1}) we get
\begin{equation}
\sum_{n\geq1}\frac{R_G\left(n\right)}{n^{s}}=\frac{s\left(s+1\right)}{\lambda^{s+1}}
\int_{1/\lambda}^{+\infty}x^{-s-2}\left(\psi*\psi\right)\left(\lambda x\right) \, \dx x
=s\left(s+1\right)\int_{1}^{+\infty}u^{-s-2}\left(\psi*\psi\right)\left(u\right) \, \dx u\label{eq:dirichlet}
\end{equation}
and note that, since $\left(\psi*\psi\right)\left(x\right)\ll x^{3}$, we
have that the integral is convergent. Hence, inserting (\ref{eq:conv})
in (\ref{eq:dirichlet}) and making trivial computations we get
\begin{align*}
&\sum_{n\geq1}\sum_{m\geq1}\frac{\Lambda\left(n\right)\Lambda\left(m\right)}{\left(m+n\right)^{s}}
=\frac{s\left(s+1\right)}{6\left(s-2\right)}-2\sum_{\rho}\frac{s\left(s+1\right)}{\rho\left(\rho+1\right)\left(\rho+2\right)\left(s-\rho-1\right)}
\\
&\quad
+\sum_{\rho_{1}}\sum_{\rho_{2}}\frac{\Gamma\left(\rho_{1}\right)\Gamma\left(\rho_{2}\right)s\left(s+1\right)}{\Gamma\left(\rho_{1}+\rho_{2}+2\right)\left(s-\rho_{1}-\rho_{2}\right)}
+O\left(\left|s\right|\left|s+1\right|\int_{1}^{+\infty}u^{-\text{Re}(s)} \, \dx u\right)
\end{align*}
and clearly the integral in the $O$-term is convergent if $\text{Re}(s)>1$.
Hence, this formula, under the assumption of RH, proves a meromorphic
continuation of the function $\Phi(s)=\sum_{n\geq1}\frac{R_G\left(n\right)}{n^{s}}$
to the half plane $\text{Re}(s)>1$ with poles at $s=2$ and $s=1+\rho$.
Note that the double series $\sum_{\rho_{1}}\sum_{\rho_{2}}\frac{\Gamma\left(\rho_{1}\right)\Gamma\left(\rho_{2}\right)}{\Gamma\left(\rho_{1}+\rho_{2}+2\right)}$
is absolutely convergent and, under RH, the function $\frac{s\left(s+1\right)}{s-\rho_{1}-\rho_{2}}$
has no poles in the region $\text{Re}(s)>1$. Also observe that the
residues of $\frac{s+1}{6\left(s-2\right)}N^{s}$ at $s=2$ and $-2\sum_{\rho}\frac{\left(s+1\right)N^{s}}{\rho\left(\rho+1\right)\left(\rho+2\right)\left(s-\rho-1\right)}$
at $s=\rho+1$ return exactly the first two terms of the asymptotic
formula of $\sum_{n\leq N}R_G\left(n\right)$.

\section{weighted averages of representations of a prime and an integer power}

In this section we show that with our method we can easily prove
asymptotic formulae for weighted averages for the number of representations
of a prime and an integer power. As far as we know,
the study of this average was done only in the case of squares (see
\cite{LANZAC1,CANGAMZAC}) and, for technical reasons, it is not possible
to study the problem of an arbitrary integer power with the proposed
techniques. Now we will show that, with our method, the general problem
can be easily attacked.

Indeed, let us fix an integer $\ell\geq2$, a function $f$ that verifies
the hypothesis of Proposition \ref{prop:abel sum} with $a=0$ and
the functions
\begin{equation*}
  r_{\ell}\left(n\right) := \begin{cases}
  1, & \text{if } n=k^{\ell}\text{\ for some }k\in\mathbb{N}^{+}\\
  0, & \text{otherwise,}
  \end{cases}
  \qquad\text{and}\qquad
  R_{\ell}\left(x\right):=\sum_{n\leq x}r_{\ell}\left(n\right).
\end{equation*}
Then, taking $\lambda>0,$ we have
\[
  \Sigma_{\ell}(f; \lambda; [0,b])
  :=
\sum_{n\leq\lambda b}\sum_{m\leq\lambda b-n}\Lambda\left(n\right)r_{\ell}\left(m\right)f\left(\frac{n+m}{\lambda}\right)=\frac{1}{\lambda}\int_{0}^{b}f^{\prime\prime}\left(w\right)\left(\psi*R_{\ell}\right)\left(\lambda w\right) \, \dx w.
\]
So the first step is to find the explicit formula for $\left(\psi*R_{\ell}\right)\left(\lambda w\right)$.
We introduce some notation:
\begin{align*}
  \M_0^{\ell}(x)
  &:=
  \frac{\ell^{2}}{2\ell^{2}+3\ell+1} x^{2+1/\ell}, \\
  \M_1^{\ell}(x)
  &:=
  -\frac{\Gamma\left(1/\ell\right)}{\ell}
  \sum_{\rho}
    \frac{x^{\rho+1/\ell+1}\Gamma\left(\rho\right)}
         {\Gamma\left(\rho+2+1/\ell\right)}
  =
  -\Gamma\left(1 + \frac1\ell \right)
  x^{1 + 1 / \ell} Z_x(1 + 1 / \ell), \\
  \M_2^{\ell}(x)
  &:=
  -
  \frac{x^{2}}{2}
  +
  x^{2+1/\ell}
  \sum_{n\geq1}
    \sum_{m=1}^{2}\dbinom{2}{m}\left(-1\right)^{m}
    \sum_{k=1}^{\ell m}\dbinom{\ell m}{k}
      \frac{k! \sin\left(2\pi nx^{1/\ell}+\frac{k\pi}{2}\right)}
           {\left(2\pi nx^{1/\ell}\right)^{k+1}}.
\end{align*}
As we show below, $\M_0^{\ell}$ and $\M_1^{\ell}$ arise respectively
from the main term and the sum over zeros in the classical explicit
formula for the function $\psi$.
$\M_2^{\ell}$ yields a secondary term stemming from the average of the
fractional part, and also an oscillatory term of lower order of magnitude.
The oscillatory terms in $\M_2^{\ell}$ are meaningful only in the
conditional case.

\begin{thm}\label{prime_integ}
We have
\[
  (\psi * R_{\ell})(x)
  =
  \M_0^{\ell}(x) + \M_1^{\ell}(x) + \M_2^{\ell}(x)
  + O \bigl( \Delta(x)  \bigr),
\]
where $\Delta(x) = x^2 \exp\bigl\{ -C \sqrt{\log(x)} \bigr\}$
unconditionally and $\Delta(x) = x^{3/2}$ if we assume RH.
\end{thm}

\begin{proof}
Let $x>2$ be a real number. Now, since $R_{\ell}\left(x\right)=\left\lfloor x^{1/\ell}\right\rfloor =x^{1/\ell}-\left\{ x^{1/\ell}\right\} $,
we note that
\begin{align*}
\left(\psi*R_{\ell}\right)\left(x\right) & =\int_{0}^{x}\psi\left(t\right)R_{\ell}\left(x-t\right) \, \dx t=\int_{0}^{x}\psi\left(t\right)\left(x-t\right)^{1/\ell} \, \dx t-\int_{0}^{x}\psi\left(t\right)\left\{ \left(x-t\right)^{1/\ell}\right\}  \, \dx t.
\end{align*}
We can integrate the explicit formula of $\psi(x)$ term by term (see
\cite{RAMSAO}, Lemma 4), hence we have 
\begin{align*}
\int_{0}^{x}\psi\left(t\right)\left(x-t\right)^{1/\ell} \, \dx t & =\int_{2}^{x}t\left(x-t\right)^{1/\ell} \, \dx t-\sum_{\rho}\frac{1}{\rho}\int_{2}^{x}t^{\rho}\left(x-t\right)^{1/\ell} \, \dx t\\
  &\qquad
  -\int_{2}^{x}\left(\log\left(2\pi\right)+\frac{1}{2}\log\left(1-\frac{1}{t^{2}}\right)\right)\left(x-t\right)^{1/\ell} \, \dx t.
\end{align*}
Obviously, 
\begin{align*}
\notag
  \int_{2}^{x}t\left(x-t\right)^{1/\ell} \, \dx t & =x^{2+1/\ell}\frac{\ell^{2}}{2\ell^{2}+3\ell+1}+O\left(x^{1/\ell}\right)
  = \M_0^{\ell}(x) + O(x^{1/\ell}),\\
\notag
  -\sum_{\rho}\frac{1}{\rho}\int_{2}^{x}t^{\rho}\left(x-t\right)^{1/\ell} \, \dx t
   & =\left(x-2\right)^{1/\ell}\sum_{\rho}\frac{2^{\rho+1}}{\rho\left(\rho+1\right)}-\frac{1}{\ell}\sum_{\rho}\frac{1}{\rho\left(\rho+1\right)}\int_{2}^{x}t^{\rho+1}\left(x-t\right)^{1/\ell-1} \, \dx t\\
\notag
  & =-\frac{\Gamma\left(1/\ell\right)}{\ell}\sum_{\rho}\frac{x^{\rho+1/\ell+1}\Gamma\left(\rho\right)}{\Gamma\left(\rho+2+1/\ell\right)}+O_{\ell}\left(x^{1/\ell}\right)
  = \M_1^{\ell}(x) + O(x^{1/\ell}),
  \end{align*}
  \begin{align}
\label{contr-0}
  \log\left(2\pi\right)\int_{2}^{x}\left(x-t\right)^{1/\ell} \, \dx t
  & =
  \log(2 \pi) \frac{\ell}{\ell+1} x^{1+1/\ell} + O(x^{1/\ell})
  =
  O_{\ell}\left(x^{1/\ell+1}\right)\\
\notag
  -\frac{1}{2}\int_{2}^{x}\log\left(1-\frac{1}{t^{2}}\right)\left(x-t\right)^{1/  \ell} \, \dx t
  &\ll
  x^{1/\ell} \int_{2}^{x} \frac{\dx t}{t^2} 
  =
  O_{\ell}\left(x^{1/\ell}\right).
\end{align}
The term in \eqref{contr-0} is meaningful only when $\ell = 1$.
Combined with the contributions from $\M_0^1$ and $\M_1^1$, it
essentially yields a proof of the explicit formula for $\psi_1(x)$.
Now let us consider 
\begin{equation}
\label{MT-frac}
  \int_{0}^{x}\psi\left(t\right)\left\{ \left(x-t\right)^{1/\ell}\right\}  \, \dx t
  =
  \int_{0}^{x}t\left\{ \left(x-t\right)^{1/\ell}\right\}  \, \dx t
  +
  O\left(\int_{0}^{x}\left|\psi\left(t\right)-t\right| \, \dx t\right).
\end{equation}
Using the Fourier series of the fractional part
\[
\left\{ x\right\} =\frac{1}{2}-\sum_{n\geq1}\frac{\sin\left(2\pi nx\right)}{n\pi}
\]
where $x$ is not an integer (see, e.g, \cite{Tit}, formula (2.1.7))
and recalling that a Fourier series of a periodic and piecewise continuous
function can be integrated termwise (see, e.g, \cite{Fol}, Theorem
2.4), after an integration by parts and a change of variables we get
\begin{align*}
\int_{0}^{x}t\left\{ \left(x-t\right)^{1/\ell}\right\}  \, \dx t & 
=\frac{x^{2}}{2}
-\sum_{n\geq1}\frac{1}{n\pi}\int_{0}^{x}t\sin\left(2\pi n\left(x-t\right)^{1/\ell}\right) \, \dx t\\
 & =
 \frac{x^{2}}{2}-\sum_{n\geq1}\frac{1}{\ell}\int_{0}^{x}t^{2}\cos\left(2\pi n\left(x-t\right)^{1/\ell}\right)\left(x-t\right)^{1/\ell-1} \, \dx t\\
 & =\frac{x^{2}}{2}-x^{2+1/\ell}\sum_{n\geq1}\int_{0}^{1}\left(1-u^{\ell}\right)^{2}\cos\left(2\pi nx^{1/\ell}u\right) \, \dx u.
\end{align*}
Put $2\pi nx^{1/\ell}=a_{n}\left(x\right)=a$. Since $\ell$ is an
integer, we have (section 2.633, formula 2, page 215 of \cite{GraRyz})
that
\[
\sum_{n\geq1}\int_{0}^{1}\left(1-u^{\ell}\right)^{2}\cos\left(au\right) \, \dx u
=\sum_{k=0}^{2}\dbinom{2}{k}\left(-1\right)^{k}\sum_{n\geq1}\int_{0}^{1}u^{\ell k}\cos\left(au\right) \, \dx u
\]
\[
=\sum_{n\geq1}\sum_{m=0}^{2}\dbinom{2}{m}\left(-1\right)^{m}\sum_{k=0}^{\ell m}\dbinom{\ell m}{k}\frac{k!}{a^{k+1}}\left[\sin\left(a+\frac{k\pi}{2}\right)-\delta_{\ell m,k}\sin\left(\frac{k\pi}{2}\right)\right]
\]
\[
=\sum_{n\geq1}\sum_{m=1}^{2}\dbinom{2}{m}\left(-1\right)^{m}\sum_{k=1}^{\ell m}\dbinom{\ell m}{k}\frac{k!}{a^{k+1}}\left[\sin\left(a+\frac{k\pi}{2}\right)-\delta_{\ell m,k}\sin\left(\frac{k\pi}{2}\right)\right]
\]
where $\delta_{i,j}$ is the Kronecker delta and the last identity
follows from the fact the addends with $k=0$, $m=0,1,2$ cancel each
other out. As a consequence, we deduce that the series is absolutely convergent.
We remark that
\begin{align*}
  \sum_{n \geq 1}
    \sum_{m=1}^{2}
      \dbinom{2}{m} \left(-1\right)^{m + 1}
      \frac{\sin\left(\frac{\ell m \pi}{2}\right)}
           {\left(2\pi nx^{1/\ell}\right)^{\ell m + 1}}
  &=
  \sum_{n \geq 1}
    \Bigl(
      \frac{2\sin\left(\frac{\ell \pi}{2}\right)}
           {\left(2\pi nx^{1/\ell}\right)^{\ell + 1}}
      -
      \frac{\sin\left(\ell \pi\right)}
           {\left(2\pi nx^{1/\ell}\right)^{2 \ell + 1}}
    \Bigr) \\
  &=
  \frac{2\sin\left(\frac{\ell \pi}{2}\right)}
       {\left(2\pi x^{1/\ell}\right)^{\ell + 1}}
  \sum_{n \geq 1}
    \frac1{n^{\ell + 1}}
  =
  \frac{\sin\left(\frac{\ell \pi}{2}\right)}
       {2^\ell \pi^{\ell + 1} x^{1 + 1/\ell}}
  \zeta(\ell + 1).
\end{align*}
Hence
\begin{align*}
  \int_{0}^{x}t
    \left\{ \left(x-t\right)^{1/\ell}\right\}  \, \dx t
  &=
  \frac{x^{2}}{2}
  -
  x^{2+1/\ell}
  \sum_{n\geq1}
    \sum_{m=1}^{2}\dbinom{2}{m}\left(-1\right)^{m}
    \sum_{k=1}^{\ell m}\dbinom{\ell m}{k}
      \frac{k! \sin\left(2\pi nx^{1/\ell}+\frac{k\pi}{2}\right)}
           {\left(2\pi nx^{1/\ell}\right)^{k+1}} \\
  &\qquad-
  x
  \frac{\sin\left(\frac{\ell \pi}{2}\right) \zeta(\ell + 1)}
       {2^\ell \pi^{\ell + 1}}
  =
  -\M_2^{\ell}(x) + O(x).
\end{align*}
In particular note that
\begin{align*}
x^{2+1/\ell}&
\sum_{n\geq1}\sum_{m=1}^{2}\dbinom{2}{m}\left(-1\right)^{m}
\sum_{k=1}^{\ell m}\dbinom{\ell m}{k}\frac{k!\sin\left(2\pi nx^{1/\ell}+\frac{k\pi}{2}\right)}{\left(2\pi nx^{1/\ell}\right)^{k+1}}
\\
&\quad
\ll x^{2+1/\ell}\sum_{n\geq1}\left(\sum_{k=1}^{\ell}\dbinom{\ell}{k}\frac{k!}{n^{k+1}x^{(k+1)/\ell}}
+\sum_{k=1}^{2\ell}\dbinom{2\ell}{k}\frac{k!}{n^{k+1}x^{(k+1)/\ell}}\right)\ll_{\ell}x^{2-1/\ell}.
\end{align*}

Note also that for $\ell=2$ the integral $\int_{0}^{1}\left(1-u^{\ell}\right)^{2}\cos\left(2\pi nx^{1/\ell}u\right) \, \dx u$
can be written in terms of the Bessel $J$ function, due to the integral
representation
\[
\int_{0}^{1}\left(1-s^{2}\right)^{\nu-\frac{1}{2}}\cos\left(us\right) \, \dx s
=\pi^{1/2}J_{\nu}\left(u\right)\Gamma\left(\nu+\frac{1}{2}\right)\left(\frac{u}{2}\right)^{-\nu},\,u\in\mathbb{C},\,\text{Re}(\nu)>-\frac{1}{2}
\]
and this is coherent with the results obtained with the Ces\`aro averages
of the functions that count the number of representations of integers
as a sum of prime powers and squares (see \cite{CAN1,CAN2,CANGAMZAC,LANZAC1},
even if in these results the case $k=1$ is not achieved). 

It remains to consider the error term in \eqref{MT-frac}:
unconditionally we have
\[
  \int_{0}^{x}\left|\psi\left(t\right)-t\right| \, \dx t 
  \ll
  1 + \int_{2}^{x}t\exp\left(-C\sqrt{\log\left(t\right)}\right) \, \dx t
  \ll
  x^{2}\exp\left(-C\sqrt{\log\left(x\right)}\right).
\]
In the conditional case we write
\[
  \int_{0}^{x}\left|\psi\left(t\right)-t\right| \, \dx t 
  \ll
  \Bigl( \int_0^x \vert \psi(t) - t \vert^2 \, \dx t \int_0^x \dx t \Bigr)^{1/2}
  \ll
  x^{3/2},
\]
by an estimate due to Cram\'er (see Theorem 13.5 of Montgomery and Vaughan
\cite{MV}).
This concludes the proof.
\end{proof}

Hence, if we put $S(x)=S_{n,m,\ell,k}\left(x\right):=\sin\left(2\pi nx^{1/\ell}+\frac{k\pi}{2}\right)$ and $I = [2 / \lambda, b]$, 
using Theorem \ref{prime_integ} and recalling \eqref{def-I}, we obtain
\begin{cor}
\label{cor-HL-gen}
\begin{align*}
  \Sigma_{\ell}(f; \lambda; [0,b])
  &=
  \frac{\ell^{2}\lambda^{1+1/\ell}}{2\ell^{2}+3\ell+1}
  \I(2 + 1/\ell; I)
  - \frac12 \lambda
  \I(2; I) \\
  &\qquad
  -\frac{\Gamma\left(1/\ell\right)}{\ell}\sum_{\rho}\frac{\lambda^{\rho+1/\ell}\Gamma\left(\rho\right)}{\Gamma\left(\rho+2+1/\ell\right)}
  \I(\rho + 1 + 1/\ell; I) \\
  &\qquad
  +\sum_{n\geq1}\sum_{m=1}^{2}\sum_{k=1}^{\ell m}\dbinom{2}{m}\dbinom{\ell m}{k}\frac{\left(-1\right)^{m}k!\lambda^{1-k/\ell}}{\left(2\pi n\right)^{k+1}}\int_{2/\lambda}^{b}f^{\prime\prime}\left(w\right)w^{2-k/\ell}S\left(\lambda w\right) \, \dx w\\
  &\qquad
  +O\left(\lambda\int_{2/\lambda}^{b}\left|f^{\prime\prime}\left(w\right)\right|w^{2}\exp\left(-C\sqrt{\log\left(\lambda w\right)}\right) \, \dx w\right).
\end{align*}
\end{cor}

We remark that in the special case $\ell = 2$, $\lambda = 1$, $b = N$
and $f(x) = (1-x/N)^k$ for $x \in [0, 1]$ and $0$ elsewhere, the first
three terms in the development above agree with the corresponding
terms in the main result in Languasco and Zaccagnini \cite{LANZAC1}.

\section{a distributional approach to additive problems}

In this part, we want to show that our approach shows
how some techniques in distribution theory can be used to address
the study of asymptotic formulas for binary averages of arithmetical functions.
Let $\mathcal{D}\left(\mathbb{R}_{0}^{+}\right)$ be the Schwartz
space of the test functions consisting of smooth compactly supported
functions with its usual topology, and denote with $\mathcal{D}^{\prime}\left(\mathbb{R}_{0}^{+}\right)$
the space of distributions. We recall some notions and results about
the concept of quasiasymptotic behavior.

\begin{defn}
A distribution $g\in\mathcal{D}^{\prime}\left(\mathbb{R}_{0}^{+}\right)$
has quasiasymptotic behavior at infinity in $\mathcal{D}^{\prime}\left(\mathbb{R}_{0}^{+}\right)$
with respect to a real function $\rho$, which is assumed to be positive
and measurable near infinity, if for every test function $\phi$ we
have that the limit
\begin{equation}
\lim_{\lambda\rightarrow+\infty}\left\langle \frac{g\left(\lambda x\right)}{\rho\left(\lambda\right)},\phi\left(x\right)\right\rangle =\left\langle \gamma(x),\phi\left(x\right)\right\rangle \label{eq:quasiasym def}
\end{equation}
exists and is finite, where $\left\langle f,g\right\rangle :=\int f(x)g(x)\,\dx x$.
It is common to use the symbol
\[
g\left(\lambda x\right)\sim\rho\left(\lambda\right)\gamma(x)\quad\text{as}\:\lambda\rightarrow+\infty\,\text{in}\,\mathcal{D}^{\prime}\left(\mathbb{R}_{0}^{+}\right)
\]
to indicate the quasiasymptotic behavior at infinity. (see, e.g, \cite{ESTKAN,PILSTATAK}).
\end{defn}

\begin{lem}
With the previous assumptions on $\rho(\lambda)$ we get that if (\ref{link quasiasymp})
holds and $\gamma\neq0$ then $\rho$ is a regular varying function
at infinity, that is
\[
\lim_{x\rightarrow+\infty}\frac{\rho\left(ax\right)}{\rho(x)}=a^{\alpha}
\]
for every $a>0$ and for some $\alpha\in\mathbb{R}$, which is called
index of regular variation (see \cite{KAR1,KAR2}), and $\gamma$
is a homogeneous distribution, that is

\[
\left\langle \gamma\left(ax\right),\phi\left(x\right)\right\rangle =a^{\alpha}\left\langle \gamma\left(x\right),\phi\left(x\right)\right\rangle ,\,\forall a>0,\,\forall\phi\in\mathcal{D}\left(\mathbb{R}_{0}^{+}\right)
\]
having degree of homogeneity equal to the index of regular variation.
\end{lem}

For a proof of a more general result, see \cite{ESTKAN0}.
It is well-known that this approach can be used to prove the Prime
Number Theorem (PNT); indeed, the quasiasymptotics behavior
\[
\lim_{\lambda\rightarrow+\infty}\psi^{\prime}\left(\lambda x\right)=\lim_{\lambda\rightarrow+\infty}\sum_{n\geq1}\Lambda\left(n\right)\delta\left(\lambda x-n\right)=1\text{\ensuremath{\qquad\text{ in }}\ensuremath{\mathcal{D}^{\prime}\left(\mathbb{R}_{0}^{+}\right)}}
\]
where, clearly, $\psi^{\prime}\left(\lambda x\right)$ is interpreted
as a distribution and $\delta(x)$ is the Dirac delta, can be used
to prove that $\psi(x)\sim x$ as $x\rightarrow+\infty$ (see \cite{ESTVIN}).
Now, equation \ref{eq:abel conseq} holds if we take the weight
$f$ in $\mathcal{D}\left(\mathbb{R}_{0}^{+}\right)$, hence this
formula can be reformulated in terms of quasiasymptotics behavior.
Indeed, taking $G\left(x\right)=\sum_{n\leq x}g(n)$ and $\lambda>0$,
we have
\[
G^{\prime}\left(x\right)=\sum_{n\geq1}g(n)\delta\left(x-n\right)
\]
and since $G(x)$ is a locally integrable function with support in
$\mathbb{R}_{0}^{+}$, it is not difficult to prove that the convolution
of two summatory functions $\left(G_{1}*G_{2}\right)$ is well defined
and also it is the convolution, in the sense of distributions,
of $\partial^{2}\left(G_{1}*G_{2}\right)=\left(G_{1}^{\prime}*G_{2}^{\prime}\right)$.
For more information about the convolution of distributions see,
e.g., \cite{VLA}, chapter $1$, sections $3,4$. So, we have that
the main formula \ref{eq:abel conseq} can be interpreted as
\[
\sum_{n\geq1}\sum_{m\geq1}g_{1}(n)g_{2}(m)\delta\left(\lambda x-n-m\right)=\left(G_{1}^{\prime}*G_{2}^{\prime}\right)\left(\lambda x\right)
\]
clearly, in the sense of distributions. So, if we have sufficient
information about the convolution $\left(G_{1}*G_{2}\right)$, as
a consequence we have a simple method to evaluate the asymptotic behavior
of $\sum_{n\leq x}\mathcal{G}\left(n\right)$.

We recall the following theorem:

\begin{thm}
\label{link quasiasymp}Let $S_{+}^{\prime}$ the space of tempered
distributions whose supports lie on the positive half-line and $C,\alpha>0$.
The generalized function $g(x)\in S_{+}^{\prime}$ has a quasiasymptotic
behaviour 
\[
g\left(\lambda x\right)\sim C\frac{\left(\lambda x\right)_{+}^{\alpha}}{\Gamma\left(\alpha+1\right)}\quad\text{as}\,\lambda\rightarrow+\infty\,\text{in}\,\mathcal{D}^{\prime}\left(\mathbb{R}_{0}^{+}\right)
\]
where $(x)_{+}=\max\left\{ x,0\right\} $, if and only if there exists
a non-negative integer $N>-\alpha-1$ such that 
\[
g^{(-N)}(x)\sim C\frac{(x)_{+}^{\alpha+N}}{\Gamma\left(\alpha+N+1\right)}
\]
in the ordinary sense as $x\rightarrow+\infty.$
\end{thm}

For a proof see, e.g., \cite{Vin}, Proposition $1.8$ or \cite{DROZAV},
Theorem 1.
Now we are able to formalize the following theorem.

\begin{thm}
Let $\alpha,C>0.$ If
\[
\left(G_{1}*G_{2}\right)\left(x\right)\sim C\frac{x^{\alpha+2}}{\Gamma\left(\alpha+3\right)},\text{as }\,x\rightarrow+\infty,
\]
then
\[
\sum_{n\leq x}\mathcal{G}\left(n\right)\sim C\frac{x^{\alpha+1}}{\Gamma\left(\alpha+2\right)}.
\]
\end{thm}

\begin{proof}
From Theorem \ref{link quasiasymp} and \ref{eq:abel conseq} we deduce that
\[
\sum_{n\geq1}\sum_{m\geq1}g_{1}(n)g_{2}(m)\delta\left(\lambda x-n-m\right)\sim C\frac{\left(\lambda x\right)_{+}^{\alpha}}{\Gamma\left(\alpha+1\right)}\quad\text{as}\,\lambda\rightarrow+\infty\,\text{in}\,\mathcal{D}^{\prime}\left(\mathbb{R}_{0}^{+}\right)
\]
which means that
\[
\lim_{\lambda\rightarrow+\infty}\frac{1}{\lambda^{\alpha+1}}\sum_{n\geq1}\sum_{m\geq1}g_{1}(n)g_{2}(m)\phi\left(\frac{m+n}{\lambda}\right)=\frac{C}{\Gamma(\alpha+1)}\int_{0}^{+\infty}x^{\alpha}\phi\left(x\right)dx
\]
for every test function $\phi(x).$ Hence, following the approach
of \cite{ESTVIN}, if we take $\varepsilon>0$ and
$\phi_{1}(x),\phi_{2}(x)\in\mathcal{D}\left(\mathbb{R}_{0}^{+}\right)$
such that $0\leq\phi_{j}(x)\leq1,\,j=1,2$, $\text{supp}\phi_{1}(x)\subseteq(0,1]$,
$\text{supp}\phi_{2}(x)\subseteq(0,1+\varepsilon]$, $\phi_{1}(x)=1,\,\forall x\in[\varepsilon,1-\varepsilon]$
and $\phi_{2}(x)=1,\forall x\in[\varepsilon,1]$, we get
\[
\limsup_{\lambda\rightarrow+\infty}\frac{1}{\lambda^{\alpha+1}}\sum_{n\leq\lambda}\sum_{m_{1}+m_{2}=n}g_{1}(m_{1})g_{2}(m_{2})\leq\limsup_{\lambda\rightarrow+\infty}\left(\frac{1}{\lambda^{\alpha+1}}\sum_{n\leq\varepsilon\lambda}\sum_{m_{1}+m_{2}=n}g_{1}(m_{1})g_{2}(m_{2})\right.
\]
\[
\left.+\frac{1}{\lambda^{\alpha+1}}\sum_{n\geq1}\sum_{m\geq1}g_{1}(n)g_{2}(m)\phi_{2}\left(\frac{m+n}{\lambda}\right)\right)\leq M(\alpha)\varepsilon^{\alpha+1}+\frac{C}{\Gamma(\alpha+2)}
\]
where $M(\alpha)$ is a suitable function depending only on $\alpha$,
and
\begin{align*}
\liminf_{\lambda\rightarrow+\infty}\frac{1}{\lambda^{\alpha+1}}\sum_{n\leq\lambda}\sum_{m_{1}+m_{2}=n}g_{1}(m_{1})g_{2}(m_{2}) & \geq \liminf_{\lambda\rightarrow+\infty}\frac{1}{\lambda^{\alpha+1}}\sum_{n\geq1}\sum_{m\geq1}g_{1}(n)g_{2}(m)\phi_{1}\left(\frac{m+n}{\lambda}\right) \\
& \geq\frac{C\left(1-2\varepsilon\right)}{\Gamma(\alpha+2)}
\end{align*}
and this completes the proof.
\end{proof}

For example, taking $g_{1}\left(n\right)=g_{2}\left(n\right)=\Lambda\left(n\right)$,
using the trivial estimation $\left(\psi*\psi\right)\left(x\right)=\frac{x^{3}}{6}+o\left(x^{3}\right)$
we immediately get
\[
\sum_{n\leq x}R_{G}\left(n\right)=\sum_{n\leq x}\sum_{m_{1}+m_{2}=n}\Lambda\left(m_{1}\right)\Lambda\left(m_{2}\right)\sim\frac{x^{2}}{2}
\]
as $x\rightarrow+\infty.$
\begin{rem}
The proposed application is just a simple observation. We strongly
believe that the language of the theory of distributions could be very
natural and useful for some problems in analytic number theory.
\end{rem}

\section{Acknowledgements}
The first author is a member of the Gruppo Nazionale per l'Analisi Matematica, la 
Probabilit\`a e le loro Applicazioni (GNAMPA) of the Istituto Nazionale di Alta 
Matematica (INdAM). We thank Prof. Alessandro Languasco for several discussions
about this work and Prof. Jasson Vindas for a conversation about the last chapter of this paper.

\begin{tabular}{l}
Marco Cantarini \\
Dipartimento di Matematica e Informatica \\
Universit\`a di Perugia \\
Via Vanvitelli, 1 \\
06123, Perugia, Italia \\
email (MC): \texttt{marco.cantarini@unipg.it} \\
Alessandro Gambini\\
Dipartimento di Matematica Guido Castelnuovo\\
Sapienza Universit\`a di Roma\\ 
Piazzale Aldo Moro, 5\\
00185 Rome, Italy\\
email (AG): \texttt{alessandro.gambini@uniroma1.it} \\
Alessandro Zaccagnini \\
Dipartimento di Scienze, Matematiche, Fisiche e Informatiche \\
Universit\`a di Parma \\
Parco Area delle Scienze 53/a \\
43124 Parma, Italia \\
email (AZ): \texttt{alessandro.zaccagnini@unipr.it}
\end{tabular}

\end{document}